\documentclass{article}%
\usepackage{amsfonts}
\usepackage{amsmath}
\usepackage{amssymb}
\usepackage{comment}
\usepackage{caption}
\RequirePackage{doi}
\hypersetup{linktoc = all}
\usepackage{tikz}
\usetikzlibrary{positioning}
\usetikzlibrary{arrows,decorations.markings}
\usetikzlibrary{shapes.geometric, arrows}
\usepackage{arydshln}
\usepackage{graphicx}
\usepackage{float}
\usepackage[square,numbers]{natbib}
\providecommand{\U}[1]{\protect\rule{.1in}{.1in}}

\newtheorem{theorem}{Theorem}

\newtheorem{corollary}[theorem]{Corollary}

\newtheorem{example}[theorem]{Example}

\newtheorem{lemma}[theorem]{Lemma}

\newtheorem{problem}[theorem]{Problem}
\newtheorem{proposition}[theorem]{Proposition}
\newtheorem{remark}[theorem]{Remark}

\newenvironment{proof}[1][Proof]{\noindent\textbf{#1.} }{\ \rule{0.5em}{0.5em}}
\makeatother

\usepackage{titlesec}
\titleformat{\paragraph}
{\normalfont\normalsize\bfseries}{\theparagraph}{1em}{}
\titlespacing*{\paragraph}
{0pt}{3.25ex plus 1ex minus .2ex}{1.5ex plus .2ex}
\definecolor{AnthonyComments}{HTML}{00BFFF}

\begin{document}
\sloppy

\title{Matrix monotonicity and concavity of the principal pivot transform}
\author{Kenneth Beard\footnote{Email: klathr2@lsu.edu}\\Louisiana State University\\Department of Mathematics\\Baton Rouge, LA, USA 70803\\\\ Aaron Welters\footnote{Email: awelters@fit.edu}\\Florida Institute of Technology\\Department of Mathematical Sciences\\Melbourne, FL, USA 32901}
\date{\today}

\maketitle

\begin{abstract}
We prove the (generalized) principal pivot transform is matrix monotone, in the sense of the L{\"o}wner ordering, under minimal hypotheses. This improves on the recent results of J.\ E.\ Pascoe and R.\ Tully-Doyle, \textit{Monotonicity of the principal pivot transform}, Linear Algebra Appl. 662 (2022) in two ways. First, we use the ``generalized" principal pivot transform, where matrix inverses in the classical definition of the principal pivot transform are replaced with Moore-Penrose pseudoinverses. Second, the hypotheses on matrices for which monotonicity holds is relaxed and, in particular, we find the weakest hypotheses possible for which it can be true. We also prove the principal pivot transform is a matrix convex function on positive semi-definite matrices that have the same kernel (and, in particular, on positive definite matrices). Our proof is a corollary of a minimization variational principle for the principal pivot transform.
\end{abstract}



\section{Introduction}

Suppose that $A\in M_n(\mathbb{F})$, where $\mathbb{F}=\mathbb{R}$ or $\mathbb{F}=\mathbb{C}$, is partitioned into a $2\times 2$ block matrix
\begin{gather}
    A=[A_{ij}]_{i,j=1,2}=\begin{bmatrix}
    A_{11} & A_{12}\\
    A_{21} & A_{22}
    \end{bmatrix},\label{notat:AMatrixPartitioned}
\end{gather}
and $A_{ij}\in M_{n_i\times n_j}(\mathbb{F})$ for each $i,j=1,2$.
Then the (generalized) principal pivot transform \cite{86AM, 00MT, 14KBa, 14KBb, 15SB, 22KJ} of $A$ with respect to the $(2,2)$-block $A_{22}$ is the matrix $\operatorname{ppt}(A)\in M_n(\mathbb{F})$ defined by  
\begin{gather}
    \operatorname{ppt}(A)=\begin{bmatrix}
    A/A_{22} & A_{12}A_{22}^{+}\\
    -A_{22}^{+}A_{21} & A_{22}^{+}
    \end{bmatrix},
\end{gather}
where $A/A_{22}\in M_{n_1}(\mathbb{F})$ is the (generalized) Schur complement \cite{74CH, 88BM, 05FZ} of $A$ with respect to the $(2,2)$-block $A_{22}$ defined by
\begin{gather}
    A/A_{22}=A_{11}-A_{12}A_{22}^{+}A_{21},
\end{gather}
and $A_{22}^+$ denotes the Moore-Penrose pseudoinverse \cite{03BG, 09CM} of $A_{22}$ (in particular, if $A_{22}$ is invertible then $A_{22}^+=A_{22}^{-1}$).  Note one could consider instead the generalized principal pivot transform and Schur complement of $A$ with respect to the $(1,1)$-block $A_{11}$, and our results will still hold under trivial modifications (cf.\ \cite[Sec.\ 4.4, Lemmas 38 and 40]{21AWa} and \cite[pp.\ 20-21, Sec.\ 1.1]{05FZ}).

Alternatively, we can define as in \cite{22PT}, the map $\operatorname{PPT}(A)\in M_{n}(\mathbb{F})$ as the``principal pivot transformation" of $A$ (or following the terminology in \cite{16PS}, it would be appropriate to refer to $\operatorname{PPT}(A)$ as the ``symmetric principal pivot transform" although we do not do so in this paper), where
\begin{gather}
    \operatorname{PPT}(A)=J\operatorname{ppt}(A)=\begin{bmatrix}
    A/A_{22} & A_{12}A_{22}^{+}\\
    A_{22}^{+}A_{21} & -A_{22}^{+}
    \end{bmatrix}
\end{gather}
and the signature matrix $J\in M_n(\mathbb{F})$ is defined to be the block matrix, partitioned conformally to the block structure of $A=[A_{ij}]_{i,j=1,2}$, by
\begin{gather}
    J=\begin{bmatrix}
    I_{n_1} & 0\\
    0 & -I_{n_2}
    \end{bmatrix}.
\end{gather}

In this paper, we study the matrix monotonicity and convexity of the map $\operatorname{PPT}(\cdot)=J\operatorname{ppt}(\cdot)$ in the sense of the L{\"o}wner ordering $\leq $, i.e., for two self-adjoint matrices $A^*=A,B^*=B\in M_n(\mathbb{F})$ we write $A\leq B$ if $B-A$ is a positive semi-definite matrix, i.e., $0\leq B-A$. Section \ref{sec:MainResultsMotivAndRelevWork} contains our statement of main results on this monotonicity (namely, Theorem \ref{thm:IntroMainThmPaperMonoGPPT}) and convexity (namely, Theorem \ref{thm:JPPTIsConcaveOnPosSemiDefMatricesWithSameKer}) of $\operatorname{PPT}(\cdot)=J\operatorname{ppt}(\cdot)$. In that same section, we will provide a proof of the convexity result using a minimization variational principle (namely, Theorem \ref{thm:PptMinVarPrinc}) for this map, but delay the proof of the monotonicity until Sec.\ \ref{sec:PrfMonoPPT}.

\section{Main results}\label{sec:MainResultsMotivAndRelevWork}
We begin by stating one of our main results of this paper, Theorem \ref{thm:IntroMainThmPaperMonoGPPT}, on the matrix monotonicity of the map $\operatorname{PPT}(\cdot)=J\operatorname{ppt}(\cdot)$ in the sense of the L{\"o}wner ordering $\leq $, but postpone the proof of this theorem until Sec.\ \ref{sec:PrfMonoPPT}.

The following will clarify some of the hypotheses (and notation) that we use. First, if $A^*=A=[A_{ij}]_{i,j=1,2}\in M_n(\mathbb{F})$, then $A_{ij}^*=A_{ji}$ for $i,j=1,2$ and $[J\operatorname{ppt}(A)]^*=J\operatorname{ppt}(A).$ If, in addition, $B^*=B=[B_{ij}]_{i,j=1,2}\in M_n(\mathbb{F})$, i.e., $B$ is partitioned conformal to $A=[A_{ij}]_{i,j=1,2}$ in (\ref{notat:AMatrixPartitioned}) [so that $A_{ij}, B_{ij}\in M_{n_i\times n_j}(\mathbb{F})$ for each $i,j=1,2$], and $A\leq B$ then $A_{22}\leq B_{22}$. Given these hypotheses, the next theorem gives us necessary and sufficient conditions to have $J\operatorname{ppt}(A)\leq J\operatorname{ppt}(B)$ in terms of the Moore-Penrose pseudoinverses $(A_{22})^+, (B_{22})^+$ of $A_{22}, B_{22},$ respectively, and the rank of the convex combination of the matrices $A_{22}$ and $B_{22}$, i.e., the function $ \operatorname{rank} [(1-t)A_{22}+tB_{22}]$ for $t\in [0,1]$ (see also Lemmas \ref{lem:MonotMoorePenosePseudoInv} and \ref{lem:MonotMPInvSpectralCharacterization}).

\begin{theorem}\label{thm:IntroMainThmPaperMonoGPPT}
     If $A^*=A=[A_{ij}]_{i,j=1,2}, B^*=B=[B_{ij}]_{i,j=1,2}\in M_n(\mathbb{F})$, and
\begin{gather}
        A\leq B
\end{gather}
then the following statements are equivalent:
\begin{itemize}
    \item[(a)] $J\operatorname{ppt}(A)\leq J\operatorname{ppt}(B)$,
    \item[(b)] $(B_{22})^+\leq (A_{22})^+$,
    \item[(c)] $ \operatorname{rank} [(1-t)A_{22}+tB_{22}]$ is constant for all $t\in [0,1]$.
\end{itemize}
Moreover, if any of these statements is true then
\begin{gather}
    A/A_{22}\leq B/B_{22}.\label{IntroMainThmPaperMonoSchurCompl}
\end{gather}
\end{theorem}

This theorem is new at least in regard to the characterization of the monotonicity of the map $\operatorname{PPT}(\cdot)=J\operatorname{ppt}(\cdot)$ in the L{\"o}wner ordering $\leq $. The statement $(c)\Rightarrow (a)$ can be seen as a generalization of \cite[Theorem 1.1.(2)]{22PT}, where they prove this statement (using our notation) in the special case $\operatorname{rank} [(1-t)A_{22}+tB_{22}]=n_2$, i.e., $(1-t)A_{22}+tB_{22}$ is invertible, for all $t\in [0,1]$.  Another comparable result is that the conclusion of our theorem in regard to the monotonicity of the Schur complement following from our hypotheses and statement $(b)$ is a generalization of \cite[Lemma 2.3]{01CW}, since our result implies theirs, but our hypotheses are much weaker. As such, our theorem may be useful in their setting in which they considered monotonicity and comparison results for discrete-time algebraic Riccati equations \cite{95LR, 03AF} and associated Riccati operators (see, for instance, Theorems 1.1 and 1.2 and Lemma 3.1 in \cite{01CW}) based on the known connection between Riccati equations and Schur complements \cite{85GT, 90AT, 90AM, 01CW}. It should also be noted that monotonicity of the Schur complement for \textit{positive semi-definite} matrices, i.e., $0\leq A\leq B$ implies $0\leq A/A_{22}\leq B/B_{22}$, is well-known \cite{70EH, 71WA, 74CH, 00LM}. 

The following example illustrate our theorem (cf.\ Sec.\ 1, last para.\ in \cite{22PT}).
\begin{example}
Let $\mathbb{F}=\mathbb{R}$ or $\mathbb{F}=\mathbb{C}$. Consider the function $J\operatorname{ppt}(\cdot)$ on the following $2\times 2$ block matrices $A^*=A=[A_{ij}]_{i,j=1,2}, B^*=B=[B_{ij}]_{i,j=1,2}\in M_2(\mathbb{F})$,
\begin{gather}
    A=\left[\begin{array}{c;{2pt/2pt}c}
			A_{11}  & A_{12}  \\\hdashline[2pt/2pt]
			A_{21}  & A_{22} 
		\end{array}\right]=\left[\begin{array}{c;{2pt/2pt}c}
			0  & 0  \\\hdashline[2pt/2pt]
			0  & -1 
		\end{array}\right], B=\left[\begin{array}{c;{2pt/2pt}c}
			A_{11}  & A_{12}  \\\hdashline[2pt/2pt]
			A_{21}  & A_{22} 
		\end{array}\right]=\left[\begin{array}{c;{2pt/2pt}c}
			0  & 0  \\\hdashline[2pt/2pt]
			0  & 1 
		\end{array}\right],\\
  J\operatorname{ppt}(A)=\left[\begin{array}{c;{2pt/2pt}c}
			A/A_{22}  & A_{12}A_{22}^{+}  \\\hdashline[2pt/2pt]
			A_{22}^{+}A_{21}  & -A_{22}^{+} 
		\end{array}\right]=\left[\begin{array}{c;{2pt/2pt}c}
			0  & 0  \\\hdashline[2pt/2pt]
			0  & 1
		\end{array}\right],\\
    J\operatorname{ppt}(B)=\left[\begin{array}{c;{2pt/2pt}c}
			B/B_{22}  & B_{12}B_{22}^{+}  \\\hdashline[2pt/2pt]
			B_{22}^{+}B_{21}  & -B_{22}^{+} 
		\end{array}\right]=\left[\begin{array}{c;{2pt/2pt}c}
			0  & 0  \\\hdashline[2pt/2pt]
			0  & -1
		\end{array}\right].
\end{gather}
Then $A\leq B$ so that the hypotheses of Theorem \ref{thm:IntroMainThmPaperMonoGPPT} are satisfied and , in particular, we have $J\operatorname{ppt}(A)\nleq J\operatorname{ppt}(B)$, $B_{22}^+=B_{22}^{-1}=\begin{bmatrix}
    1
\end{bmatrix}\nleq \begin{bmatrix}
    -1
\end{bmatrix}=A_{22}^{-1}=A_{22}^{+}$,
and $\operatorname{rank} [(1-t)A_{22}+tB_{22}]$ is not constant on $[0,1]$ since at $t=0$ we have $\operatorname{rank} A_{22}=1$, whereas at $t=1/2$ we have $\operatorname{rank} [(1/2)A_{22}+(1/2)B_{22}]=0$.
\end{example}

We next establish the following result that $\operatorname{PPT}(A)=J\operatorname{ppt}(A)$ is itself a Schur complement. Although the result is easily proved as we will show, and is based on Proposition 42 in \cite{21AWa} and Lemma 74 in \cite{22KB}, it is useful as it allows us to apply the theory of Schur complements \cite{74CH,05FZ} to prove many important results on the principal pivot transform as we shall see.
\begin{proposition}\label{prop:GPPTIsASchurCompl}
    If $A=[A_{ij}]_{i,j=1,2}\in M_n(\mathbb{F})$ then $J\operatorname{ppt}(A)$ is the Schur complement
    \begin{gather}
    J\operatorname{ppt}(A)=\hat{A}/\hat{A}_{22},
\end{gather}
where $\hat{A}=[\hat{A}_{ij}]_{i,j=1,2}\in M_{n+n_2}(\mathbb{F})$ is the block matrix
\begin{gather}
    \hat{A}=\left[\begin{array}{c;{2pt/2pt}c}
			\hat{A}_{11}  & \hat{A}_{12}  \\\hdashline[2pt/2pt]
			\hat{A}_{21}  & \hat{A}_{22} 
		\end{array}\right]=\left[\begin{array}{c c;{2pt/2pt} c}
			A_{11} & 0 & A_{12} \\ 
			0 & 0 & -A_{22}^+A_{22}\\ \hdashline[2pt/2pt]
			A_{21} & -A_{22}A_{22}^+ & A_{22}
		\end{array}\right]
\end{gather}
    and $A_{22}A_{22}^+$ and $A_{22}^+A_{22}$ are the orthogonal projections of $\mathbb{F}^{n_2}$ onto $\operatorname{ran}(A_{22})$ and $\operatorname{ran}(A_{22}^*)$, respectively.
\end{proposition}
\begin{proof}
    By a direct calculation using block multiplication, we have
\begin{align*}
    \hat{A}/\hat{A}_{22}&=\hat{A}_{11}-\hat{A}_{12}\hat{A}_{22}^{+}\hat{A}_{21}=\begin{bmatrix}
            A_{11} & 0  \\ 
			0 & 0 & 
    \end{bmatrix}-\begin{bmatrix}
            A_{12}\\
            -A_{22}^+A_{22}
    \end{bmatrix}A_{22}^+\begin{bmatrix}
            A_{21} & -A_{22}A_{22}^+
    \end{bmatrix}\\
    &=\begin{bmatrix}
            A_{11}-A_{12}A_{22}^+A_{21} & A_{12}A_{22}^+A_{22}A_{22}^+  \\ 
			A_{22}^+A_{22}A_{22}^+A_{21} & -A_{22}^+A_{22}A_{22}^+A_{22}A_{22}^+ & 
    \end{bmatrix}=\begin{bmatrix}
            A/A_{22} & A_{12}A_{22}^+  \\ 
			A_{22}^+A_{21} & -A_{22}^+ & 
    \end{bmatrix}\\
    &=J\operatorname{ppt}(A).
\end{align*}
\end{proof}

For instance, from this proposition we can easily derive, using the theory of Schur complements, the following corollary that $\operatorname{PPT}(\cdot)=J\operatorname{ppt}(\cdot)$ maps matrices with positive semi-definite imaginary part to matrices with positive semi-definite imaginary part, which is a generalization of Theorem 1.1.(1) in \cite{22PT}. Recall, that if $A\in M_n(\mathbb{F})$ then its imaginary part $\operatorname{Im}(A)\in M_n(\mathbb{C})$ is defined as the self-adjoint matrix $\operatorname{Im}(A)=\frac{1}{2i}(A-A^*)$. We begin with a lemma.
\begin{lemma}\label{lem:ImPartSchurComplementFormula}
     If $A=[A_{ij}]_{i,j=1,2}\in M_n(\mathbb{F})$ and $A_{22}$ is an EP matrix, i.e., $A_{22}A_{22}^+=A_{22}^+A_{22},$
    then
    \begin{gather}
        A/A_{22}=\begin{bmatrix}
                I_{n_1}\\
                -A_{22}^+A_{21}
        \end{bmatrix}^*A\begin{bmatrix}
                I_{n_1}\\
                -A_{22}^+A_{21}
        \end{bmatrix},\\
        \operatorname{Im}(A/A_{22})=\begin{bmatrix}
                I_{n_1}\\
                -A_{22}^+A_{21}
        \end{bmatrix}^*\operatorname{Im}(A)\begin{bmatrix}
                I_{n_1}\\
                -A_{22}^+A_{21}
        \end{bmatrix}.
    \end{gather}
    In particular, if $0\leq \operatorname{Im}A$ then $A_{22}$ is an EP matrix and
    \begin{gather}
        0\leq \operatorname{Im}(A/A_{22}).
    \end{gather}
\end{lemma}
\begin{proof}
    Assume the hypotheses. Then
\begin{gather*}
    \begin{bmatrix}
                I_{n_1}\\
                -A_{22}^+A_{21}
        \end{bmatrix}^*A\begin{bmatrix}
                I_{n_1}\\
                -A_{22}^+A_{21}
        \end{bmatrix}=\begin{bmatrix}
                I_{n_1} & (-A_{22}^+A_{21})^*
        \end{bmatrix}\begin{bmatrix}
    A_{11} & A_{12}\\
    A_{21} & A_{22}
    \end{bmatrix}\begin{bmatrix}
                I_{n_1}\\
                -A_{22}^+A_{21}
        \end{bmatrix}\\
    =A/A_{22}-A_{21}^*[(I_{n_2}-A_{22}^+A_{22})A_{22}^+]^*A_{21}=A/A_{22}.
\end{gather*}
The proof of the lemma now follows immediately from this and the fact that if $0\leq \operatorname{Im}A$ then $0\leq \operatorname{Im}(A_{22})$ implying $A_{22}$ is an EP matrix.
\end{proof}

\begin{corollary}
 If $A=[A_{ij}]_{i,j=1,2}\in M_n(\mathbb{F})$ and $A_{22}$ is an EP matrix then
    \begin{gather}
        J\operatorname{ppt}(A)=\begin{bmatrix}
                I_{n}\\
                -\hat{A}_{22}^+\hat{A}_{21}
        \end{bmatrix}^*\hat{A}\begin{bmatrix}
                I_{n}\\
                -\hat{A}_{22}^+\hat{A}_{21}
        \end{bmatrix},\\
        \operatorname{Im}[J\operatorname{ppt}(A)]=\begin{bmatrix}
                I_{n_1} & 0\\
                -A_{22}^+A_{21} & A_{22}^+
        \end{bmatrix}^*\operatorname{Im}(A)\begin{bmatrix}
                I_{n_1} & 0\\
                -A_{22}^+A_{21} & A_{22}^+
        \end{bmatrix}.
    \end{gather}
     In particular, if $0\leq \operatorname{Im}(A)$ then 
     \begin{gather}
         0\leq \operatorname{Im}[J\operatorname{ppt}(A)].
     \end{gather}
\end{corollary}
\begin{proof}
    The proof follows immediately from Proposition \ref{prop:GPPTIsASchurCompl} and Lemma \ref{lem:ImPartSchurComplementFormula} by observing that
    \begin{gather*}
        \operatorname{Im}(\hat A)=C^*\operatorname{Im}(A)C,\;
        C\begin{bmatrix}
                I_{n}\\
                -\hat{A}_{22}^+\hat{A}_{21}
        \end{bmatrix}
        =\begin{bmatrix}
                I_{n_1} & 0\\
                -A_{22}^+A_{21} & A_{22}^+
        \end{bmatrix},
    \end{gather*}
    where $C\in M_{n\times (n+n_2)}(\mathbb{F})$ is the block matrix defined by
    \begin{gather*}
        C=\left[\begin{array}{ccc}
			I_{n_1}  & 0 & 0  \\
			0  & 0& I_{n_2} 
			\end{array}\right].
    \end{gather*}
\end{proof}

Another instance of the usefulness of Proposition \ref{prop:GPPTIsASchurCompl} begins with recalling the following well-known variational principle for the Schur complement \cite{00LM,05FZ, 22KW} with a short proof of it.
\begin{lemma}\label{lem:SchurCompMinPrinc}
If $A^*=A=[A_{ij}]_{i,j=1,2}\in M_n(\mathbb{F})$, $0\leq A_{22}$, and $\ker A_{22}\subseteq \ker A_{12}$ then $A/A_{22}$ is the unique self-adjoint matrix satisfying the minimization principle:
	\begin{gather}
		\left(x_1,A/A_{22}x_1\right)=\min_{x_2\in \mathbb{F}^{n_2}}\left( \begin{bmatrix}x_1\\x_2 \end{bmatrix}, A\begin{bmatrix}x_1\\x_2 \end{bmatrix} \right),
	\end{gather}
	for all $x_1\in \mathbb{F}^{n_1}$. Furthermore, for each $x_1\in \mathbb{F}^{n_1}$, the set of minimizers is
	\begin{gather}
	    \{-A_{22}^+A_{21}x_1\}+\ker{A_{22}}.
	\end{gather}
\end{lemma}
\begin{proof}
Assume the hypotheses. Then by Lemma \ref{lem:GenLem6_1In01CW} it follows that, for any $(x_1,x_2)\in\mathbb{F}^{n_1}\times \mathbb{F}^{n_2}$,
\begin{gather*}
    \left (\begin{bmatrix}x_1\\x_2\end{bmatrix}A\begin{bmatrix}x_1\\x_2 \end{bmatrix}\right)=\left(\begin{bmatrix}
	            I_{n_1} & 0\\
	            A_{22}^+A_{21} & I_{n_2}
	    \end{bmatrix}\begin{bmatrix}x_1\\x_2\end{bmatrix},\begin{bmatrix}
	            A/A_{22} & 0\\
	            0 & A_{22}
	    \end{bmatrix}\begin{bmatrix}
	            I_{n_1} & 0\\
	            A_{22}^+A_{21} & I_{n_2}
	    \end{bmatrix}\begin{bmatrix}x_1\\x_2\end{bmatrix}\right)\\
	    =(x_1,A/A_{22}x_1)+((A_{22}^+A_{21}x_1+x_2),A_{22}(A_{22}^+A_{21}x_1+x_2))\geq (x_1,A/A_{22}x_1)
\end{gather*}
with equality if and only if $A_{22}^+A_{21}x_1+x_2\in \ker A_{22}$ if and only if $x_2\in \{-A_{22}^+A_{21}x_1\}+\ker{A_{22}}.$
\end{proof}

Using this we get immediately the next theorem on a minimization variational principle for $\operatorname{PPT}(\cdot)=J\operatorname{ppt}(\cdot)$, whose minimizers solve the following problem based on the next proposition.
\begin{problem}\label{prob:LinearEqsProbForA}
Let $A=[A_{ij}]_{i,j=1,2}\in M_n(\mathbb{F})$. Given $(x_1,y_2)\in \mathbb{F}^{n_1}\times\mathbb{F}^{n_2},$ find all $(y_1,x_2)\in \mathbb{F}^{n_1}\times\mathbb{F}^{n_2}$ such that
\begin{gather}
    A\begin{bmatrix}x_1\\x_2 \end{bmatrix}=\begin{bmatrix}y_1\\y_2 \end{bmatrix}.\label{eq:ProbLinearEqsProbForA}
\end{gather}
\end{problem}
\begin{proposition}[Solution of Problem \ref{prob:LinearEqsProbForA}]\label{prop:SolnOfLinearEqsProbForA}
If $A=[A_{ij}]_{i,j=1,2}\in M_n(\mathbb{F})$, $\operatorname{ran} A_{21}\subseteq \operatorname{ran} A_{22}$, and $\ker A_{22}\subseteq \ker A_{12}$ then for each $(x_1,y_2)\in \mathbb{F}^{n_1}\times\mathbb{F}^{n_2}$ and for $(y_1,x_2)\in \mathbb{F}^{n_1}\times\mathbb{F}^{n_2}$,
\begin{gather}
    A\begin{bmatrix}x_1\\x_2 \end{bmatrix}=\begin{bmatrix}y_1\\y_2 \end{bmatrix} \iff \left\{\begin{array}{l}
         x_2\in \{-A_{22}^+A_{21}x_1+A_{22}^+y_2\}+\ker{A_{22}},\\
        y_1=A/A_{22}x_1+A_{12}A_{22}^+y_2
    \end{array}\right.
\end{gather}
with a particular solution $(y_1,x_2^0)\in \mathbb{F}^{n_1}\times\mathbb{F}^{n_2}$ given by 
\begin{gather}
    J\operatorname{ppt}(A)\begin{bmatrix}x_1\\y_2 \end{bmatrix}=\begin{bmatrix}y_1\\-x_2^0 \end{bmatrix},\;\;x_2^0=-A_{22}^+A_{21}x_1+A_{22}^+y_2.
\end{gather}
\end{proposition}
\begin{proof}
    Assume the hypotheses. Let $(x_1,y_2)\in \mathbb{F}^{n_1}\times\mathbb{F}^{n_2}$. Then by Lemma \ref{lem:GenLem6_1In01CW} and since $A_{22}A_{22}^+A_{21}=A_{21}$ (by the hypothesis $\operatorname{ran} A_{21}\subseteq \operatorname{ran} A_{22}$), it follows that for $(y_1,x_2)\in \mathbb{F}^{n_1}\times\mathbb{F}^{n_2}$,
\begin{gather*}
    A\begin{bmatrix}x_1\\x_2 \end{bmatrix}=\begin{bmatrix}y_1\\y_2 \end{bmatrix}\iff 
     \begin{bmatrix}
	            A/A_{22} & 0\\
	            0 & A_{22}
	    \end{bmatrix}\begin{bmatrix}x_1\\A_{22}^+A_{21}x_1+x_2\end{bmatrix}=\begin{bmatrix}y_1-A_{12}A_{22}^+y_2\\y_2 \end{bmatrix}\\
     \iff  \begin{bmatrix}A/A_{22}x_1\\A_{21}x_1+A_{22}x_2\end{bmatrix}=\begin{bmatrix}y_1-A_{12}A_{22}^+y_2\\y_2 \end{bmatrix}\\ \iff \begin{bmatrix}A/A_{22}x_1\\A_{22}^+A_{21}x_1+A_{22}^+A_{22}x_2\end{bmatrix}=\begin{bmatrix}y_1-A_{12}A_{22}^+y_2\\A_{22}^+y_2 \end{bmatrix}.
\end{gather*}
The proof of the proposition now follows immediately from this using the fact that $A_{22}^+A_{22}$ is the orthogonal projection of $\mathbb{F}^{n_2}$ onto $\operatorname{ran}(A_{22}^*)$ and as such its kernel is $\ker{A_{22}}$ so that in particular, 
\begin{gather*}
    J\operatorname{ppt}(A)\begin{bmatrix}x_1\\y_2 \end{bmatrix}=\begin{bmatrix}
    A/A_{22} & A_{12}A_{22}^{+}\\
    A_{22}^{+}A_{21} & -A_{22}^{+}
    \end{bmatrix}\begin{bmatrix}x_1\\y_2 \end{bmatrix}=\begin{bmatrix} A/A_{22}x_1+A_{12}A_{22}^{+}y_2\\ A_{22}^{+}A_{21}x_1-A_{22}^{+}y_2 \end{bmatrix}\\
    =\begin{bmatrix}y_1\\-A_{22}^+A_{22}x_2 \end{bmatrix}
    =\begin{bmatrix}y_1\\-A_{22}^+A_{22}(-A_{22}^+A_{21}x_1+A_{22}^+y_2) \end{bmatrix}\\
    =\begin{bmatrix}y_1\\-(-A_{22}^+A_{21}x_1+A_{22}^+y_2)\end{bmatrix}=\begin{bmatrix}y_1\\-x_2^0 \end{bmatrix}.
\end{gather*}
\end{proof}

\begin{theorem}\label{thm:PptMinVarPrinc}
	If $A^*=A=[A_{ij}]_{i,j=1,2}\in M_n(\mathbb{F})$, $0\leq A_{22}$, and $\ker A_{22}\subseteq \ker A_{12}$ then $J\operatorname{ppt}(A)$ is the unique self-adjoint matrix satisfying the minimization principle:
	\begin{gather}
		\frac{1}{2}\left( \begin{bmatrix}x_1\\y_2 \end{bmatrix},J \operatorname{ppt}(A)\begin{bmatrix}x_1\\y_2 \end{bmatrix}\right)\!=\!\min_{x_2\in \mathbb{F}^{n_2}}\!\left\{\frac{1}{2}\left( \begin{bmatrix}x_1\\x_2 \end{bmatrix}, A\begin{bmatrix}x_1\\x_2 \end{bmatrix} \right)\!-\!\operatorname{Re}\left( y_2,(A_{22}A_{22}^+)x_2\right)\right\},
	\end{gather}
	for all $(x_1,y_2)\in \mathbb{F}^{n_1}\times \mathbb{F}^{n_2}$. Furthermore, for each $(x_1,y_2)\in \mathbb{F}^{n_1}\times \mathbb{F}^{n_2}$, the set of minimizers is
	\begin{gather}
	    \{-A_{22}^+A_{21}x_1+A_{22}^+y_2\}+\ker{A_{22}}.
	\end{gather}
\end{theorem}
\begin{proof}
Assume the hypotheses. Then by Proposition \ref{prop:GPPTIsASchurCompl} we have $[J\operatorname{ppt}(A)]^*=J\operatorname{ppt}(A)=\hat A/\hat A_{22}$ with $\hat A_{22}=A_{22}^+\geq 0$ and $\ker \hat A_{22}=\ker A_{22}=\ker A_{22}\cap \ker A_{12}=\ker \begin{bmatrix} A_{12} \\  -A_{22}^+A_{22}\end{bmatrix}=\ker \hat A_{12}$. Thus, it follows immediately by Lemma \ref{lem:SchurCompMinPrinc} that $J\operatorname{ppt}(A)$ is the unique self-adjoint matrix satisfying the minimization principle:
	\begin{gather*}
		\left(\begin{bmatrix}x_1\\y_2\end{bmatrix},J\operatorname{ppt}(A)\begin{bmatrix}x_1\\y_2\end{bmatrix}\right)=\min_{x_2\in \mathbb{F}^{n_2}}\left( \begin{bmatrix}x_1\\y_2\\x_2\end{bmatrix}, \hat A\begin{bmatrix}x_1\\y_2\\x_2 \end{bmatrix} \right),
	\end{gather*}
	for all $(x_1,y_2)\in \mathbb{F}^{n_1}\times \mathbb{F}^{n_2}$ with the set of minimizers
	\begin{gather*}
	    \left\{-\hat A_{22}^+\hat A_{21}\begin{bmatrix}x_1\\y_2\end{bmatrix}\right\}+\ker{\hat A_{22}}.
	\end{gather*}
The proof of the theorem now follows immediately from this since, for any $(x_1,y_2,x_2)\in \mathbb{F}^{n_1}\times \mathbb{F}^{n_2}\times \mathbb{F}^{n_2}$,
\begin{gather*}
    -\hat A_{22}^+\hat A_{21}\begin{bmatrix}x_1\\y_2\end{bmatrix}=-A_{22}^+\begin{bmatrix}
           A_{21} & -A_{22}A_{22}^+ 
    \end{bmatrix}\begin{bmatrix}x_1\\y_2\end{bmatrix}
    =-A_{22}^+A_{21}x_1+A_{22}^+y_2,\\
    \left( \begin{bmatrix}x_1\\y_2\\x_2\end{bmatrix}, \hat A\begin{bmatrix}x_1\\y_2\\x_2 \end{bmatrix} \right)=\left( \begin{bmatrix}x_1\\x_2\end{bmatrix}, A\begin{bmatrix}x_1\\x_2 \end{bmatrix} \right)-2\operatorname{Re}\left(y_2,A_{22}A_{22}^+x_2 \right) .
\end{gather*}
\end{proof}

This theorem, which appears to be new in this context, is insightful as it indicates possible applications of the theory of the principal pivot transform (and our main theorem, i.e., Theorem \ref{thm:IntroMainThmPaperMonoGPPT}) to linear systems of saddle point type \cite{05BG} such as those occurring in classical analytic mechanics when the Legendre transform is used to derive the Hamiltonian formulation from the Lagrangian formulation, and conversely (see, for instance, \cite[Sec.\ II.A and cf.\ Eqs.\ (1), (12), (13)]{14FW} for a brief summary of the relevant theory for Lagrangians that are quadratic forms and \cite{02HG,13VA} for the general theory). This is one of our motivations for this paper.

Another motivation for this paper in connection to Theorem \ref{thm:PptMinVarPrinc} is its potential application in the theory of composites, where the principal pivot transform and such variational principles appear in the Cherkaev-Gibiansky-Milton (CGM) method (see \cite{90GM, 94CG,02AC, 02GM, 09MS, 10MW, 12RD, 17RR, 16GM}) which is used, for instance, in obtaining upper and lower bounds on effective tensors of multiphased composites with lossy inclusions.

One immediate application of Theorem \ref{thm:PptMinVarPrinc} is another of our main results, which appears to be new, that $\operatorname{PPT}(\cdot)=J\operatorname{ppt}(\cdot)$ is a concave function on positive semidefinite matrices with the same kernel and, in particular, on the set of all positive definite matrices.
\begin{theorem}\label{thm:JPPTIsConcaveOnPosSemiDefMatricesWithSameKer}
    If $A^*=A=[A_{ij}]_{i,j=1,2},B^*=B=[B_{ij}]_{i,j=1,2}\in M_n(\mathbb{F})$, $0\leq A$, $0\leq B$, and $\ker A_{22}=\ker B_{22}$ then
    \begin{gather}
        (1-t)J\operatorname{ppt}A+tJ\operatorname{ppt}B\leq J\operatorname{ppt}[(1-t)A+tB],\;\forall t\in [0,1].
    \end{gather}
\end{theorem}
\begin{proof}
    Assume the hypotheses. Then $A_{22}^*=A_{22}\geq 0, B_{22}^*=B_{22}\geq 0,$ and $\ker A_{22}=\ker B_{22}$ implies
    \begin{gather*}
        \ker[(1-t)A_{22}+tB_{22}]=\ker A_{22}=\ker B_{22},\;\forall t\in[0,1]
    \end{gather*}
    from which it follows that
    \begin{gather*}
        [(1-t)A_{22}+tB_{22}][(1-t)A_{22}+tB_{22}]^+=A_{22}A_{22}^+=B_{22}B_{22}^+,\;\forall t\in[0,1]
    \end{gather*}
    and hence
    \begin{gather*}
        A_{22}A_{22}^+=B_{22}B_{22}^+=(1-t)A_{22}A_{22}^++tB_{22}B_{22}^+,\;\forall t\in[0,1].
    \end{gather*}
    Next, the hypotheses of Theorem \ref{thm:PptMinVarPrinc} are true of $(1-t)A+tB$ for each $t\in [0,1]$ implying for all $t\in [0,1]$ and all $(x_1,y_2)\in \mathbb{F}^{n_1}\times \mathbb{F}^{n_2}$,
    \begin{gather*}
		\frac{1}{2}\left( \begin{bmatrix}x_1\\y_2 \end{bmatrix},J \operatorname{ppt}[(1-t)A+tB]\begin{bmatrix}x_1\\y_2 \end{bmatrix}\right)\\
  =\min_{x_2\in \mathbb{F}^{n_2}}\bigg\{\frac{1}{2}\left( \begin{bmatrix}x_1\\x_2 \end{bmatrix}, [(1-t)A+tB]\begin{bmatrix}x_1\\x_2 \end{bmatrix} \right)\bigg.\\
  \bigg.-\operatorname{Re}\left( y_2,[(1-t)A_{22}+tB_{22}][(1-t)A_{22}+tB_{22}]^+x_2\right)\bigg\}\\
  =\min_{x_2\in \mathbb{F}^{n_2}}\bigg\{\frac{1}{2}\left( \begin{bmatrix}x_1\\x_2 \end{bmatrix}, [(1-t)A+tB]\begin{bmatrix}x_1\\x_2 \end{bmatrix} \right)\bigg.\\
  \bigg.-\operatorname{Re}\left( y_2,[(1-t)A_{22}A_{22}^++tB_{22}B_{22}^+]x_2\right)\bigg\}\\
  \geq (1-t)\min_{x_2\in \mathbb{F}^{n_2}}\left\{\frac{1}{2}\left( \begin{bmatrix}x_1\\x_2 \end{bmatrix}, A\begin{bmatrix}x_1\\x_2 \end{bmatrix} \right)-\operatorname{Re}\left( y_2,(A_{22}A_{22}^+)x_2\right)\right\}\\
  + t\min_{x_2\in \mathbb{F}^{n_2}}\left\{\frac{1}{2}\left( \begin{bmatrix}x_1\\x_2 \end{bmatrix}, B\begin{bmatrix}x_1\\x_2 \end{bmatrix} \right)-\operatorname{Re}\left( y_2,(B_{22}B_{22}^+)x_2\right)\right\}\\
  =\frac{1}{2}\left( \begin{bmatrix}x_1\\y_2 \end{bmatrix},[(1-t)J \operatorname{ppt}A+tJ \operatorname{ppt}B]\begin{bmatrix}x_1\\y_2 \end{bmatrix}\right).
	\end{gather*}
 The proof follows now immediately from this.
\end{proof}

As an application of this theorem, we can immediate prove the following corollaries on the well-known results on the concavity of the Schur complement (see, for instance, \cite{00LM}) and of the convexity of the Moore-Penrose pseudoinverse on positive semidefinite matrices with the same kernel \cite{81DK, 85GW, 89DR, 96MP, 11KN, 18KN} (and, in particular, on the set of all positive definite matrices \cite{73MM, 79TA, 19BS}).

\begin{corollary}
    If $A^*=A=[A_{ij}]_{i,j=1,2},B^*=B=[B_{ij}]_{i,j=1,2}\in M_n(\mathbb{F})$, $0\leq A$, $0\leq B$, and $\ker A_{22}=\ker B_{22}$ then
    \begin{gather}
        (1-t) A/A_{22}+t B/B_{22}\leq [(1-t)A+tB]/[(1-t)A+tB]_{22},\;\forall t\in [0,1].
    \end{gather}
\end{corollary}
\begin{proof}
    Assume the hypotheses. Then it follows from Theorem \ref{thm:JPPTIsConcaveOnPosSemiDefMatricesWithSameKer} that
    \begin{gather*}
        (1-t) A/A_{22}+t B/B_{22}=[(1-t)J\operatorname{ppt}A+tJ\operatorname{ppt}B]_{11}\\
        \leq \{J\operatorname{ppt}[(1-t)A+tB]\}_{11}\\
        =[(1-t)A+tB]/[(1-t)A+tB]_{22},\;\forall t\in [0,1].
    \end{gather*}
\end{proof}
\begin{corollary}
    If $C^*=C,D^*=D\in M_m(\mathbb{F})$, $0\leq C$, $0\leq D$, and $\ker C=\ker D$ then
    \begin{gather}
         [(1-t)C+tD]^+\leq (1-t)C^++t D^+,\;\forall t\in [0,1].
    \end{gather}
\end{corollary}
\begin{proof}
    Assume the hypotheses. Then it follows from Theorem \ref{thm:JPPTIsConcaveOnPosSemiDefMatricesWithSameKer} with $A^*=A=[A_{ij}]_{i,j=1,2},B^*=B=[B_{ij}]_{i,j=1,2}\in M_{n}(\mathbb{F}), n=m+1$ defined by $A_{22}=C, B_{22}=D$ and $A_{ij}=B_{ij}=0$ otherwise, (so that $0\leq A$, $0\leq B$, and $\ker A_{22}=\ker B_{22}$) that we have
    \begin{gather*}
        (1-t) C^++t D^+=-[(1-t)J\operatorname{ppt}A+tJ\operatorname{ppt}B]_{22}\\
        \geq -\{J\operatorname{ppt}[(1-t)A+tB]\}_{22}=[(1-t)C+tD]^+,\;\forall t\in [0,1].
    \end{gather*}
\end{proof}

\begin{remark}
    It should be pointed out that one could also use Theorem \ref{thm:JPPTIsConcaveOnPosSemiDefMatricesWithSameKer} to give a simple proof on the matrix monotoncity of $\operatorname{PPT}(\cdot)=J\operatorname{ppt}(\cdot)$ on positive semi-definite matrices with the same kernel, but such a result is already contained in our more general result, Theorem \ref{thm:IntroMainThmPaperMonoGPPT} (cf.\ Lemma \ref{lem:MonotMoorePenosePseudoInv}).
\end{remark}

Our results may have further applications, but, of course, this requires recognizing the occurrence of the principal pivot transform in specific problems, which may not always be clear. As such, we want to briefly mention below some additional contexts where the principal pivot transform arises under various guises. 

First, as discussed in \cite{00MT, 22PT}, the principal pivot(al) transform was introduced around 1960 by A.\ W.\ Tucker \cite{60AT, 63AT} in the context of mathematical programming to understand the linear algebraic structure underlying the simplex method of G.\ B.\ Dantzig, but it is also known as the sweep operator \cite{79JG, 10KL}, exchange operator \cite{98SS}, partial inverse \cite{06NW}, and gyration operator \cite{65DH, 66DH, 72DT, 75MT} with the latter a term coined by R.\ J.\ Duffin, D.\ Hazony, and N.\ Morrison in the context of network synthesis problems involving nonreciprocal circuit elements called gyrators \cite{48BT, 63DH} (see also Sec.\ 4.4 in \cite{21AWa}). 

Second, the principal pivot transform is essentially the Potapov-Ginzburg transform \cite[see Sec.\ 2.2 and cf.\ p.\ 21, Eq.\ (2.13)]{08AD}. This transform was introduced and studied by V.\ P.\ Potapov \cite{55VP} in 1955 and Y.\ P.\ Ginzburg \cite{57YG} in 1957, in the context of the theory $J$-contractive (nonexpanding) matrix- and operator-valued functions. It has also been called the Redheffer transform \cite[see Sec.\ 17.4 and cf.\ p.\ 328, Eqs.\ (17.17), (17.18)]{10BG} after R.\ Redheffer who introduced the transformation in 1959 in the context of matrix Riccati equations and related it to the Redheffer star product \cite{59RR, 60RR,72WR}, all which play an important role in transmission-line theory \cite[Sec.\ 2.16]{08AD}, \cite{62RR, 13RR}, the electrodynamics of layered media \cite{96LL}, and Schrodinger equations on graphs \cite{01KS} whenever the scattering matrix plays a prominent role. In these context, the Redheffer transform converts a chain-scattering or transfer matrix into the scattering matrix and as such it is also sometimes called the chain scattering transform \cite{22AS} (see also \cite{08AD, 95KO, 97HK, 02TP}).

Finally, the principal pivot transform also occurs in other contexts, for instance, in connection to solving inverse problems for canonical differential
equations \cite{94KO, 12AD}, in the study of the analytic properties of the electromagnetic Dirichlet-to-Neumann (DtN)
map in layered media \cite{16CW}, and in the solvability of elliptic PDEs in terms of their boundary data \cite{10AA}.

\section{Proof of the monotonicity of principal pivot transform}\label{sec:PrfMonoPPT}

Our goal in this section is to prove Theorem \ref{thm:IntroMainThmPaperMonoGPPT} on the monotoncity of the map $\operatorname{PPT}(\cdot)=J\operatorname{ppt}(\cdot)$ in the sense of the L{\"o}wner ordering $\leq $.

We begin by considering a matrix $B\in M_n(\mathbb{F})$ partitioned conformally to the matrix $A=[A_{ij}]_{i,j=1,2}\in M_n(\mathbb{F})$ in (\ref{notat:AMatrixPartitioned}), i.e.,
\begin{gather}
    B=[B_{ij}]_{i,j=1,2}=\begin{bmatrix}
    B_{11} & B_{12}\\
    B_{21} & B_{22}
    \end{bmatrix},\label{notat:BMatrixPartitioned}
\end{gather}
and $B_{ij}\in M_{n_i\times n_j}(\mathbb{F})$ for each $i,j=1,2$.  Then
\begin{gather}
        J\operatorname{ppt}(B)-J\operatorname{ppt}(A) =\begin{bmatrix}
    B/B_{22}-A/A_{22} & B_{12}B_{22}^{+}-A_{12}A_{22}^{+}\\
    B_{22}^{+}B_{21}- A_{22}^{+}A_{21}& A_{22}^{+}-B_{22}^{+}
    \end{bmatrix},\label{formula:DiffPPTOfAAndBMatrixPartitioned}\\
    [J\operatorname{ppt}(B)-J\operatorname{ppt}(A)]/[J\operatorname{ppt}(B)-J\operatorname{ppt}(A)]_{22}\notag\\
    = B/B_{22}-A/A_{22}-(B_{12}B_{22}^+-A_{12}A_{22}^+)(A_{22}^+-B_{22}^+)^+(B_{22}^+B_{21}-A_{22}^+A_{21}).\label{formula:SchurComplDiffPPTOfAAndBMatrixPartitioned}
\end{gather}

Now, recall the well-known lemma \{originally due to A.\ Albert \cite{69AA}, see Eqs.\ (2.6) and (2.7) in \cite{74CH} as well as Theorem 1.20 and Sec.\ 6.0.4 in \cite{05FZ} on the ``Albert nonnegative definiteness conditions"\} which characterizes self-adjoint matrices that are positive semi-definite in terms of their Schur complement.
\begin{lemma}\label{lem:NecSuffCondPosMatrixUsingSchurCompl}
If $A^*=A=[A_{ij}]_{i,j=1,2}\in M_n(\mathbb{F})$ then $0\leq A$ if and only if
\begin{gather}
    0\leq A_{22},\; \ker A_{22}\subseteq \ker A_{12},\text{ and } 0\leq A/A_{22}.
\end{gather}
\end{lemma}
\begin{remark}
    It should be noted that in some statements of our results, kernel inclusions can be replaced by equivalent range inclusions. This is a consequence of the facts that if $A=[A_{ij}]_{i,j=1,2}\in M_n(\mathbb{F})$ then $\ker A_{22}\subseteq \ker A_{12}$ iff $\operatorname{ran} (A^*)_{21}\subseteq \operatorname{ran} (A^*)_{22}$; (ii) $\operatorname{ran} A_{21}\subseteq \operatorname{ran} A_{22}$ iff $\ker (A^*)_{22}\subseteq \ker (A^*)_{12}$. This follows immediately using the facts that if $i,j\in\{1,2\}$ then $(A_{ij})^*=(A^*)_{ji}$, for any $B\in M_{n_i\times n_j},(\mathbb{F})$ we have $(\ker B)^{\perp}=\operatorname{ran}(B^*)$ and $(\operatorname{ran} B)^{\perp}=\ker(B^*)$, and if $S_1,S_2$ are subspaces of $\mathbb{F}^j$ with $S_1\subseteq S_2$ then $S_2^{\perp}\subseteq S_1^{\perp}$.
\end{remark}

From this lemma we get immediately the following result which  is a key observation that helps give insight into statements $(a)$ and $(b)$ in Theorem \ref{thm:IntroMainThmPaperMonoGPPT} and the proof they are equivalent.
\begin{corollary}\label{cor:NecSuffCondMatricesForMonotJPPT}
If $A^*=A=[A_{ij}]_{i,j=1,2}, B^*=B=[B_{ij}]_{i,j=1,2}\in M_n(\mathbb{F})$ then
\begin{gather}
    J\operatorname{ppt}(A)\leq J\operatorname{ppt}(B)
\end{gather}
if and only if
\begin{gather}
    B_{22}^+\leq A_{22}^+,\;\ker (A_{22}^+-B_{22}^+)\subseteq \ker(B_{12}B_{22}^+-A_{12}A_{22}^+),\label{cor:NecSuffCondMatricesForMonotJPPT:CondI}\\
    0\leq B/B_{22}-A/A_{22}-(B_{12}B_{22}^+-A_{12}A_{22}^+)(A_{22}^+-B_{22}^+)^+(B_{22}^+B_{21}-A_{22}^+A_{21}).\label{cor:NecSuffCondMatricesForMonotJPPT:CondII}
\end{gather}
\end{corollary}
\begin{proof}
    Assume the hypotheses. Then $[J\operatorname{ppt}(A)]^*=J\operatorname{ppt}(A), [J\operatorname{ppt}(B)]^*=J\operatorname{ppt}(B)$ so that $[J\operatorname{ppt}(B)-J\operatorname{ppt}(A)]^*=J\operatorname{ppt}(B)-J\operatorname{ppt}(A)$, and the proof now follows from this immediately by Lemma \ref{lem:NecSuffCondPosMatrixUsingSchurCompl} using (\ref{formula:DiffPPTOfAAndBMatrixPartitioned}) and (\ref{formula:SchurComplDiffPPTOfAAndBMatrixPartitioned}).
\end{proof}

Now we explore some necessary and sufficient conditions for (\ref{cor:NecSuffCondMatricesForMonotJPPT:CondI}) and (\ref{cor:NecSuffCondMatricesForMonotJPPT:CondII}) to be true. To do this we need some auxiliary results. First, it is clear from Corollary \ref{cor:NecSuffCondMatricesForMonotJPPT} and the statement of Theorem \ref{thm:IntroMainThmPaperMonoGPPT} that it is important to know under what additional conditions, besides $A_{22}^*=A_{22}\leq B_{22}=B_{22}^*$, do we need in order to guarantee that $B_{22}^+\leq A_{22}^+$ is true? The following well-known statement on the monotonicity of the Moore-Penrose pseudoinverse (see, for instance, \cite[Lemma 4, Theorem 2]{90BN}, \cite[Theorem 3.8]{14BH}, \cite{96EL, 97JG}, \cite[Chap.\ 8]{10MB}) is a satisfactory answer for our purposes.
\begin{lemma}\label{lem:MonotMoorePenosePseudoInv}
    If $C^*=C, D^*=D\in M_m(\mathbb{F})$ and $C\leq D$ then
    \begin{gather}
        D^+\leq C^+ \iff \ker C=\ker D,\; i_-(C)=i_-(D),
    \end{gather}
    where $i_-(\cdot)$ denotes the number of negative eigenvalues  (counting multiplicities) of a self-adjoint matrix. In particular, if $0\leq C\leq D$ then $D^+\leq C^+$ if and only if $\ker C=\ker D$.
\end{lemma}

Next, we need the following extension of \cite[Lemma 6.1]{01CW}. In our lemma, formula (\ref{Cond1AitkenBlockDiagFormula}) with (\ref{AitkenBlockDiagFormulaForBlocksWZ}) may appropriately be a called a generalized Aitken block-diagonalization formula \{cf.\ Eqs.\ (0.9.1) and (6.0.20) in Secs.\ 0.9 and 6.0.4, respectively, in  \cite{05FZ}\}.
\begin{lemma}\label{lem:GenLem6_1In01CW}
   If $A=[A_{ij}]_{i,j=1,2}\in M_n(\mathbb{F})$ then
   \begin{gather}
       \begin{bmatrix}
       I_{n_1} & -X\\
       0 & I_{n_2}
       \end{bmatrix}A\begin{bmatrix}
       I_{n_1} & 0\\
       -Y & I_{n_2}
       \end{bmatrix}=\begin{bmatrix}
       W & 0\\
       0 & Z
       \end{bmatrix}\label{Cond1AitkenBlockDiagFormula}
   \end{gather}
   for some $X\in M_{n_1\times n_2}(\mathbb{F}), Y\in M_{n_2\times n_1}(\mathbb{F}), Z\in M_{n_2}(\mathbb{F}), W\in M_{n_1}(\mathbb{F})$
   if and only if 
   \begin{gather}
       \ker A_{22}\subseteq \ker A_{12}\text{ and }\operatorname{ran}A_{21}\subseteq \operatorname{ran}A_{22},\label{Cond2AitkenBlockDiagFormula}
   \end{gather}
   in which case
   \begin{gather}
       Z=A_{22}, W=A/A_{22}.\label{AitkenBlockDiagFormulaForBlocksWZ}
   \end{gather}
\end{lemma}
\begin{proof}
    Assume the hypotheses, where $A_{ij}\in M_{n_i\times n_j}(\mathbb{F})$ for $i,j=1,2$. First, using the fact that 
    \begin{gather*}
     A_{22}^+A_{22}A_{22}^+=A_{22}^+,   
    \end{gather*}
    it follows by block multiplication that
    \begin{gather*}
    \begin{bmatrix}
       I_{n_1} & -A_{12}A_{22}^+\\
       0 & I_{n_2}
       \end{bmatrix}\begin{bmatrix}
       A_{11} & A_{12}\\
       A_{21} & A_{22}
       \end{bmatrix}\begin{bmatrix}
       I_{n_1} & 0\\
       -A_{22}^+A_{21} & I_{n_2}
       \end{bmatrix}\notag\\
       =\begin{bmatrix}
       A/A_{11} & A_{12}-A_{12}A_{22}^+A_{22}\\
       A_{21}-A_{22}A_{22}^+A_{21} & A_{22}
       \end{bmatrix}.\label{formula:GenAitkenBlkDiagInABlkForm}
   \end{gather*}
   Second, we have
   \begin{gather*}
       \ker A_{22}\subseteq \ker A_{12} \iff A_{12}-A_{12}A_{22}^+A_{22}=0.
   \end{gather*}
   Third, we have
   \begin{gather*}
       \operatorname{ran}A_{21}\subseteq \operatorname{ran}A_{22} \iff A_{21}-A_{22}A_{22}^+A_{21}=0.
   \end{gather*}
   This proves that
       \begin{gather*}
       \begin{bmatrix}
       I_{n_1} & -A_{12}A_{22}^+\\
       0 & I_{n_2}
       \end{bmatrix}\begin{bmatrix}
       A_{11} & A_{12}\\
       A_{21} & A_{22}
       \end{bmatrix}\begin{bmatrix}
       I_{n_1} & 0\\
       -A_{22}^+A_{21} & I_{n_2}
       \end{bmatrix}
       =\begin{bmatrix}
       A/A_{11} & 0\\
       0 & A_{22}
    \end{bmatrix}.\label{formula:KeyGenAitkenBlkDiagInABlkForm}
   \end{gather*}
   ($\Leftarrow$): Suppose the inclusions (\ref{Cond2AitkenBlockDiagFormula}) hold. Then it follows from these facts that \ref{Cond1AitkenBlockDiagFormula} is true, where we can take $X=A_{12}A_{22}^+, Y=A_{22}^+A_{21}, Z=A_{22},W=A/A_{22}$. ($\Rightarrow$): Conversely, suppose (\ref{Cond1AitkenBlockDiagFormula}) holds for some $X\in M_{n_1\times n_2}(\mathbb{F}), Y\in M_{n_2\times n_1}(\mathbb{F}), Z\in M_{n_2}(\mathbb{F}), W\in M_{n_1}(\mathbb{F})$. Then
    \begin{gather*}
        A= \begin{bmatrix}
       A_{11} & A_{12}\\
       A_{21} & A_{22}
       \end{bmatrix}
       =\begin{bmatrix}
       I_{n_1} & X\\
       0 & I_{n_2}
       \end{bmatrix}\begin{bmatrix}
       W & 0\\
       0 & Z
       \end{bmatrix}\begin{bmatrix}
       I_{n_1} & 0\\
       Y & I_{n_2}
       \end{bmatrix}\\
       =\begin{bmatrix}
       W & XZ\\
       0 & Z
       \end{bmatrix}\begin{bmatrix}
       I_{n_1} & 0\\
       Y & I_{n_2}
       \end{bmatrix}\\
       =\begin{bmatrix}
       W+XZY & XZ\\
       ZY & Z
       \end{bmatrix}
    \end{gather*}
    implying
    \begin{gather*}
        A_{22}=Z, A_{12}=XZ,  A_{21}=ZY, \\
        W=A_{11}-XZY=A_{11}-XZZ^+ZY=A_{11}-A_{12}A_{22}^+A_{21}=A/A_{22},
    \end{gather*}
    where we have used the fact that 
        \begin{gather*}
     A_{22}A_{22}^+A_{22}=A_{22}.   
    \end{gather*}
    From these equalities it follows  immediately that (\ref{Cond2AitkenBlockDiagFormula}) and (\ref{AitkenBlockDiagFormulaForBlocksWZ}) are true.
\end{proof}

The next proposition is the key result we need to prove that $(a)\Leftrightarrow (b)$ in Theorem \ref{thm:IntroMainThmPaperMonoGPPT}.
\begin{proposition}\label{prop:KerRanConditionForJPPTMonoResult}
    If $A=[A_{ij}]_{i,j=1,2}, B=[B_{ij}]_{i,j=1,2}\in M_n(\mathbb{F})$ with $\ker A_{22}=\ker B_{22}$, $\operatorname{ran} A_{22}=\operatorname{ran} B_{22}$, $\ker (B_{22}-A_{22})\subseteq\ker (B_{12}-A_{12})$, and $\operatorname{ran} (B_{21}-A_{21})\subseteq\operatorname{ran} (B_{22}-A_{22})$ then
    \begin{gather}
        \ker(A_{22}^+-B_{22}^+)\subseteq \ker(B_{12}B_{22}^+-A_{12}A_{22}^+),\label{Step1KerRanConditionForJPPTMonoResult}\\
        \operatorname{ran}(B_{22}^+B_{21}-A_{22}^+A_{21})\subseteq\operatorname{ran}(A_{22}^+-B_{22}^+).\label{Step2KerRanConditionForJPPTMonoResult}
    \end{gather}
    Furthermore,
    \begin{gather}
    (A_{12}A_{22}^+-B_{12}B_{22}^+)(A_{22}^+-B_{22}^+)^+(A_{22}^+A_{21}-B_{22}^+B_{21})\notag\\
    =(A_{12}A_{22}^+-B_{12}B_{22}^+)[A_{22}+A_{22}(B_{22}-A_{22})^+A_{22}](A_{22}^+A_{21}-B_{22}^+B_{21}).\label{Step3KerRanConditionForJPPTMonoResult}
    \end{gather}
    Moreover,
    \begin{align}
        (B-A)/(B-A)_{22}&=B/B_{22}-A/A_{22}\notag\\
        &\quad-(B_{12}B_{22}^+-A_{12}A_{22}^+)(A_{22}^+-B_{22}^+)^+(B_{22}^+B_{21}-A_{22}^+A_{21}).\label{formula:EqDiffSchurComplAndDiffPPT}
    \end{align}
\end{proposition}
\begin{proof}
        Assume the hypotheses. Then as $\ker (B_{22}-A_{22})\subseteq\ker (B_{12}-A_{12})$, it follows that
    \begin{gather}
    \ker (B_{22}-A_{22})\subseteq \ker [(B_{12}B_{22}^+A_{22}-A_{12})+B_{12}(I_{n_2}-B_{22}^+B_{22})].\label{InclusionPart1ForAitkenBlkDiagFormula}
\end{gather}
Next, as $\ker A_{22}=\ker B_{22}$, $\operatorname{ran} A_{22}=\operatorname{ran} B_{22}$, then $A_{22}A_{22}^+=B_{22}B_{22}^+,$ and $ A_{22}^+A_{22}=B_{22}^+B_{22}$ from which it follows, using also the identity $A_{22}^+A_{22}A_{22}^+=A_{22}$, that
\begin{gather*}
    [(B_{12}B_{22}^+A_{22}-A_{12})+B_{12}(I_{n_2}-B_{22}^+B_{22})]A_{22}^+=B_{12}B_{22}^+-A_{12}A_{22}^+,\\
    B_{22}A_{22}^+-A_{22}A_{22}^+=B_{22}A_{22}^+-B_{22}B_{22}^+=B_{22}(A_{22}^+-B_{22}^+).
\end{gather*}
It now follows immediately from these facts that
\begin{gather*}
    \ker(B_{22}A_{22}^+-A_{22}A_{22}^+)\subseteq \ker(B_{12}B_{22}^+-A_{12}A_{22}^+),\\
        \ker(A_{22}^+-B_{22}^+)\subseteq\ker(B_{22}(A_{22}^+-B_{22}^+))=\ker(B_{22}A_{22}^+-A_{22}A_{22}^+),
\end{gather*}
which implies
\begin{gather*}
        \ker(A_{22}^+-B_{22}^+)\subseteq \ker(B_{22}A_{22}^+-A_{22}A_{22}^+)\subseteq \ker(B_{12}B_{22}^+-A_{12}A_{22}^+).
    \end{gather*}
Now since $\operatorname{ran} (B_{21}-A_{21})\subseteq\operatorname{ran} (B_{22}-A_{22})$ implies $\ker ((B^*)_{22}-(A^*)_{22})\subseteq\ker ((B^*)_{12}-(A^*)_{12})$, then by interchanging the role of $A, B$ with $A^*, B^*$ in the proof we just gave and taking orthogonal complements, it follows that
    \begin{gather*}
        \operatorname{ran}[(A_{22}B_{22}^+B_{21}-A_{21})+(I_{n_2}-B_{22}B_{22}^+)B_{21}]\subseteq\operatorname{ran} (B_{22}-A_{22}),\label{InclusionPart2ForAitkenBlkDiagFormula}\\
        \operatorname{ran}(B_{22}^+B_{21}-A_{22}^+A_{21})\subseteq\operatorname{ran}(A_{22}^+-B_{22}^+).
    \end{gather*}
Thus, we have proven the inclusions (\ref{Step1KerRanConditionForJPPTMonoResult}) and (\ref{Step2KerRanConditionForJPPTMonoResult}) hold. Next, it follows from these inclusions that 
    \begin{gather*}
        A_{12}A_{22}^+-B_{12}B_{22}^+=(A_{12}A_{22}^+-B_{12}B_{22}^+)(A_{22}^+-B_{22}^+)^+(A_{22}^+-B_{22}^+),\label{Iden1}\\
        A_{22}^+A_{21}-B_{22}^+B_{21}=(A_{22}^+-B_{22}^+)(A_{22}^+-B_{22}^+)^+(A_{22}^+A_{21}-B_{22}^+B_{21}).\label{Iden2}
    \end{gather*}
Hence, in order to prove the identity (\ref{Step3KerRanConditionForJPPTMonoResult}), we need only prove the claim that
    \begin{gather*}
(A_{22}^+-B_{22}^+)^+(A_{22}^+-B_{22}^+)[A_{22}\!+\!A_{22}(B_{22}\!-\!A_{22})^+A_{22}](A_{22}^+-B_{22}^+)(A_{22}^+-B_{22}^+)^+\\
=(A_{22}^+-B_{22}^+)^+.
    \end{gather*}
First,
\begin{gather*}
    (A_{22}^+-B_{22}^+)[A_{22}+A_{22}(B_{22}-A_{22})^+A_{22}](A_{22}^+-B_{22}^+)\\
    =(A_{22}^+-B_{22}^+)A_{22}(A_{22}^+-B_{22}^+)+(A_{22}^+-B_{22}^+)A_{22}(B_{22}-A_{22})^+A_{22}(A_{22}^+-B_{22}^+)\\
    =(A_{22}^+A_{22}-B_{22}^+A_{22})(A_{22}^+-B_{22}^+)\\+(A_{22}^+-B_{22}^+)A_{22}(B_{22}-A_{22})^+A_{22}(A_{22}^+-B_{22}^+)\\
    =(A_{22}^+A_{22}-B_{22}^+A_{22})A_{22}^+-(A_{22}^+A_{22}-B_{22}^+A_{22})B_{22}^+\\
    +(A_{22}^+-B_{22}^+)A_{22}(B_{22}-A_{22})^+A_{22}(A_{22}^+-B_{22}^+)\\
    =(A_{22}^+A_{22}A_{22}^+-B_{22}^+A_{22}A_{22}^+)-(A_{22}^+A_{22}B_{22}^+-B_{22}^+A_{22}B_{22}^+)\\
    +(A_{22}^+-B_{22}^+)A_{22}(B_{22}-A_{22})^+A_{22}(A_{22}^+-B_{22}^+)\\
    =(A_{22}^+-B_{22}^+B_{22}B_{22}^+)-(B_{22}^+B_{22}B_{22}^+-B_{22}^+A_{22}B_{22}^+)\\
    +(A_{22}^+-B_{22}^+)A_{22}(B_{22}-A_{22})^+A_{22}(A_{22}^+-B_{22}^+)\\
    =A_{22}^+-B_{22}^+\\
    +B_{22}^+A_{22}B_{22}^+-B_{22}^++(A_{22}^+-B_{22}^+)A_{22}(B_{22}-A_{22})^+A_{22}(A_{22}^+-B_{22}^+)\\
    =A_{22}^+-B_{22}^+\\
    +B_{22}^+A_{22}B_{22}^+-B_{22}^++(A_{22}^+A_{22}-B_{22}^+A_{22})(B_{22}-A_{22})^+(A_{22}A_{22}^+-A_{22}B_{22}^+)\\
    =A_{22}^+-B_{22}^+\\
    +B_{22}^+A_{22}B_{22}^+-B_{22}^++(B_{22}^+B_{22}-B_{22}^+A_{22})(B_{22}-A_{22})^+(B_{22}B_{22}^+-A_{22}B_{22}^+)\\
    =A_{22}^+-B_{22}^+\\
    +B_{22}^+A_{22}B_{22}^+-B_{22}^++B_{22}^+(B_{22}-A_{22})(B_{22}-A_{22})^+(B_{22}-A_{22})B_{22}^+\\
    =A_{22}^+-B_{22}^++B_{22}^+A_{22}B_{22}^+-B_{22}^++B_{22}^+(B_{22}-A_{22})B_{22}^+\\
    =A_{22}^+-B_{22}^++B_{22}^+A_{22}B_{22}^+-B_{22}^++(B_{22}^+B_{22}B_{22}^+-B_{22}^+A_{22}B_{22}^+)\\
    =A_{22}^+-B_{22}^+.
\end{gather*}
Second,
\begin{gather*}
    (A_{22}^+-B_{22}^+)^+(A_{22}^+-B_{22}^+)(A_{22}^+-B_{22}^+)^+=(A_{22}^+-B_{22}^+)^+.
\end{gather*}
Thus,
\begin{gather*}
(A_{22}^+-B_{22}^+)^+(A_{22}^+-B_{22}^+)[A_{22}\!+\!A_{22}(B_{22}\!-\!A_{22})^+A_{22}](A_{22}^+-B_{22}^+)(A_{22}^+-B_{22}^+)^+\\
=(A_{22}^+-B_{22}^+)^+(A_{22}^+-B_{22}^+)(A_{22}^+-B_{22}^+)^+\\
=(A_{22}^+-B_{22}^+)^+.
\end{gather*}
This proves the claim, which proves the identity (\ref{Step3KerRanConditionForJPPTMonoResult}).

We are now ready to prove the equality (\ref{formula:EqDiffSchurComplAndDiffPPT}). First, since $A_{22}^+A_{22}=B_{22}^+B_{22}$ and $A_{22}A_{22}^+=B_{22}B_{22}^+$, it follows that
    \begin{gather*}
        \begin{bmatrix}
	            I_{n_1} & -B_{12}B_{22}^+\\
	            0 & I_{n_2}
	    \end{bmatrix}A\begin{bmatrix}
	            I_{n_1} & 0\\
	            -B_{22}^+B_{21} & I_{n_2}
	    \end{bmatrix}\\
     =\begin{bmatrix}
	            I_{n_1} & -B_{12}B_{22}^+\\
	            0 & I_{n_2}
	    \end{bmatrix}\begin{bmatrix}
	            A_{11} & A_{12}\\
	            A_{21} & A_{22}
	    \end{bmatrix}\begin{bmatrix}
	            I_{n_1} & 0\\
	            -B_{22}^+B_{21} & I_{n_2}
	    \end{bmatrix}\\
         =\begin{bmatrix}
	             A/A_{22}+A_{12}(A_{22}^+A_{21}-B_{22}^+B_{21})\\+B_{12}B_{22}^+(-A_{21} +A_{22}B_{22}^+B_{21}) & A_{12}-B_{12}B_{22}^+A_{22} \\ \\
	            A_{21}-A_{22}B_{22}^+B_{21} & A_{22}
	    \end{bmatrix}\\
     =\begin{bmatrix}
	             A/A_{22}+A_{12}A_{22}^+A_{22}(A_{22}^+A_{21}-B_{22}^+B_{21})\\
              -B_{12}B_{22}^+A_{22}(A_{22}^+A_{21} -B_{22}^+B_{21}) & A_{12}-B_{12}B_{22}^+A_{22} \\ \\
	            A_{21}-A_{22}B_{22}^+B_{21} & A_{22}
	    \end{bmatrix}\\
     =\begin{bmatrix}
	             A/A_{22}+(A_{12}A_{22}^+-B_{12}B_{22}^+)A_{22}(A_{22}^+A_{21}-B_{22}^+B_{21}) & A_{12}-B_{12}B_{22}^+A_{22} \\
	            A_{21}-A_{22}B_{22}^+B_{21} & A_{22}
	    \end{bmatrix}.
\end{gather*}
Next, by block multiplication [cf.\ (\ref{formula:GenAitkenBlkDiagInABlkForm})] we have
\begin{gather*}
    \begin{bmatrix}
	            I_{n_1} & -B_{12}B_{22}^+\\
	            0 & I_{n_2}
	    \end{bmatrix}B\begin{bmatrix}
	            I_{n_1} & 0\\
	            -B_{22}^+B_{21} & I_{n_2}
	    \end{bmatrix}
     \\
     =\begin{bmatrix}
	            B/B_{22} & B_{12}(I_{n_2}-B_{22}^+B_{22})\\
	             (I_{n_2}-B_{22}B_{22}^+)B_{21} & B_{22}
	    \end{bmatrix}.
\end{gather*}
This implies that
    \begin{gather*}
        \begin{bmatrix}
	            I_{n_1} & -B_{12}B_{22}^+\\
	            0 & I_{n_2}
	    \end{bmatrix}(B-A)\begin{bmatrix}
	            I_{n_1} & 0\\
	            -B_{22}^+B_{21} & I_{n_2}
	    \end{bmatrix}\\=\begin{bmatrix}
	            B/B_{22} & B_{12}(I_{n_2}-B_{22}^+B_{22})\\
	             (I_{n_2}-B_{22}B_{22}^+)B_{21} & B_{22}
	    \end{bmatrix}\\
     - \begin{bmatrix}
	            I_{n_1} & -B_{12}B_{22}^+\\
	            0 & I_{n_2}
	    \end{bmatrix}A\begin{bmatrix}
	            I_{n_1} & 0\\
	            -B_{22}^+B_{21} & I_{n_2}
	    \end{bmatrix}\\
     = \begin{bmatrix}
	            B/B_{22} & B_{12}(I_{n_2}-B_{22}^+B_{22})\\
	             (I_{n_2}-B_{22}B_{22}^+)B_{21} & B_{22}
	    \end{bmatrix}\\
     -\begin{bmatrix}
	             A/A_{22}+(A_{12}A_{22}^+-B_{12}B_{22}^+)A_{22}(A_{22}^+A_{21}-B_{22}^+B_{21}) & A_{12}-B_{12}B_{22}^+A_{22} \\
	            A_{21}-A_{22}B_{22}^+B_{21} & A_{22}
	    \end{bmatrix}\\
    =\begin{bmatrix}
	            [B/B_{22}-A/A_{22} & [(B_{12}B_{22}^+A_{22}-A_{12})\\
             -(A_{12}A_{22}^+-B_{12}B_{22}^+)A_{22}(A_{22}^+A_{21}-B_{22}^+B_{21})]
             & +B_{12}(I_{n_2}-B_{22}^+B_{22})]\\
	             (A_{22}B_{22}^+B_{21}-A_{21})+(I_{n_2}-B_{22}B_{22}^+)B_{21} & B_{22}-A_{22}
	    \end{bmatrix}.
    \end{gather*}
Next, the inclusions (\ref{InclusionPart1ForAitkenBlkDiagFormula}) and (\ref{InclusionPart2ForAitkenBlkDiagFormula}) together with this block form implies, by the generalized Aitken block diagonal formula 
(\ref{formula:KeyGenAitkenBlkDiagInABlkForm}) in the proof of Lemma \ref{lem:GenLem6_1In01CW}, that
\begin{gather*}
    \begin{bmatrix}
	            I_{n_1}\!\!\! & M\\
	            0 & I_{n_2}
	    \end{bmatrix}\!\!\!\begin{bmatrix}
	            I_{n_1} \!\!\!\!\!& -B_{12}B_{22}^+\\
	            0 & I_{n_2}
	    \end{bmatrix}\!\!(B\!-\!A)\!\!\begin{bmatrix}
	            I_{n_1}\!\!\!\!\!\!\!\!\! & 0\\
	            -B_{22}^+B_{21}\!\!\! & I_{n_2}
	    \end{bmatrix}\!\!\!\begin{bmatrix}
	            I_{n_1}\!\!\! & 0\\
	            N\!\!\! & I_{n_2}
	    \end{bmatrix}
     \!=\!\begin{bmatrix}
	            S\!\!\! & 0\\
	            0\!\!\! & B_{22}-A_{22}
	    \end{bmatrix},
\end{gather*}
where
\begin{gather*}
    M=-[(B_{12}B_{22}^+A_{22}-A_{12})+B_{12}(I_{n_2}-B_{22}^+B_{22})](B_{22}-A_{22})^+,\\
    N=-(B_{22}-A_{22})^+[(A_{22}B_{22}^+B_{21}-A_{21})+(I_{n_2}-B_{22}B_{22}^+)B_{21}],\\
    S=B/B_{22}-A/A_{22}-(A_{12}A_{22}^+-B_{12}B_{22}^+)A_{22}(A_{22}^+A_{21}-B_{22}^+B_{21})-S_1S_2,\\
    S_1=[(B_{12}B_{22}^+A_{22}-A_{12})+B_{12}(I_{n_2}-B_{22}^+B_{22})],\\
    S_2=(B_{22}-A_{22})^+[(A_{22}B_{22}^+B_{21}-A_{21})+(I_{n_2}-B_{22}B_{22}^+)B_{21} ].
\end{gather*}
It also follows from this and Lemma \ref{lem:GenLem6_1In01CW} that
\begin{gather*}
    S=(B-A)/(B-A)_{22}.
\end{gather*}
Next, we claim that
\begin{gather}
  (I_{n_2}-B_{22}^+B_{22})(B_{22}-A_{22})^+=0,
  (B_{22}-A_{22})^+(I_{n_2}-B_{22}B_{22}^+)=0.\label{NeededIdentity}
\end{gather}
First, we know that $I_{n_2}-B_{22}^+B_{22}$ is the orthogonal projection of $\mathbb{F}^{n_2}$ onto $\ker B_{22}$. Second, we know by fundamental properties of the Moore-Penrose pseudoinverse that
\begin{gather*}
    \ker [(B_{22}\!-\!A_{22})^+]^*\!=\!\ker [(B_{22}\!-\!A_{22})^*]^+\!=\!\ker [(B_{22}\!-\!A_{22})^*]^*\!=\!\ker (B_{22}\!-\!A_{22}).
\end{gather*}
Next, we know that $\operatorname{ran}(I_{n_2}-B_{22}^+B_{22})=\ker B_{22}$ and from our assumption that $\ker A_{22}=\ker B_{22}$, it follows that $\ker A_{22}=\ker B_{22}\subseteq \ker (B_{22}-A_{22})$. Putting all this together, we have
\begin{gather*}
    (I_{n_2}-B_{22}^+B_{22})(B_{22}-A_{22})^+=\{[(B_{22}-A_{22})^+]^*(I_{n_2}-B_{22}^+B_{22})\}^*=0^*=0.
\end{gather*}
Now, by interchanging the roll of $A_{22}, B_{22}$ with $(A^*)_{22}, (B^*)_{22}$ in the proof we just gave and using the hypothesis that $\operatorname{ran}(A_{22})=\operatorname{ran}(B_{22})$ [which is equivalent to $\ker (A^*)_{22}=\ker (B^*)_{22}$], it follows that
\begin{gather*}
  (I_{n_2}-(B^*)_{22}^+(B)^*_{22})((B^*)_{22}-(A^*)_{22})^+=0,
\end{gather*}
and then taking adjoints of this proves our claim. Thus, it follows now from the identities (\ref{NeededIdentity}), the equalities $A_{22}^+A_{22}=B_{22}^+B_{22}$ and $A_{22}A_{22}^+=B_{22}B_{22}^+$, and (\ref{Step3KerRanConditionForJPPTMonoResult}) that
\begin{gather*}
    M=-(B_{12}B_{22}^+A_{22}-A_{12})(B_{22}-A_{22})^+,\\
    N=-(B_{22}-A_{22})^+(A_{22}B_{22}^+B_{21}-A_{21}),\\
    S=B/B_{22}-A/A_{22}-(A_{12}A_{22}^+-B_{12}B_{22}^+)A_{22}(A_{22}^+A_{21}-B_{22}^+B_{21})\\
    -(B_{12}B_{22}^+A_{22}-A_{12})(B_{22}-A_{22})^+(A_{22}B_{22}^+B_{21}-A_{21})\\
    =B/B_{22}-A/A_{22}-(A_{12}A_{22}^+-B_{12}B_{22}^+)A_{22}(A_{22}^+A_{21}-B_{22}^+B_{21})\\
    -(B_{12}B_{22}^+A_{22}-A_{12})A_{22}^+A_{22}(B_{22}-A_{22})^+A_{22}A_{22}^+(A_{22}B_{22}^+B_{21}-A_{21})\\
    =B/B_{22}-A/A_{22}-(A_{12}A_{22}^+-B_{12}B_{22}^+)A_{22}(A_{22}^+A_{21}-B_{22}^+B_{21})\\
    -(B_{12}B_{22}^+-A_{12}A_{22}^+)A_{22}(B_{22}-A_{22})^+A_{22}(B_{22}^+B_{21}-A_{22}^+A_{21})\\
    =B/B_{22}-A/A_{22}\\
    -(A_{12}A_{22}^+-B_{12}B_{22}^+)[A_{22}+A_{22}(B_{22}-A_{22})^+A_{22}](A_{22}^+A_{21}-B_{22}^+B_{21})\\
    =
    B/B_{22}-A/A_{22}-(B_{12}B_{22}^+-A_{12}A_{22}^+)(A_{22}^+-B_{22}^+)^+(B_{22}^+B_{21}-A_{22}^+A_{21}).
\end{gather*}
This proves (\ref{formula:EqDiffSchurComplAndDiffPPT}) and completes the proof of the proposition.
\end{proof}

We are now ready to prove $(a)\Leftrightarrow (b)$ in Theorem \ref{thm:IntroMainThmPaperMonoGPPT}.

\begin{proof}[Proof of $(a)\Leftrightarrow (b)$ in Theorem \ref{thm:IntroMainThmPaperMonoGPPT}]
    Assume the hypotheses in Theorem \ref{thm:IntroMainThmPaperMonoGPPT}. [$(a) \Rightarrow (b)$]: Suppose $J\operatorname{ppt}(A)\leq J\operatorname{ppt}(B)$. Then $(B_{22})^+\leq (A_{22})^+$ by Corollary \ref{cor:NecSuffCondMatricesForMonotJPPT} and
    \begin{gather*}
        A/A_{22}=[J\operatorname{ppt}(A)]_{11}\leq [J\operatorname{ppt}(B)]_{11}=B/B_{22}.
    \end{gather*}
    [$(b) \Rightarrow (a)$]: Conversely, suppose $(B_{22})^+\leq (A_{22})^+$. We begin by showing that the hypotheses of Proposition \ref{prop:KerRanConditionForJPPTMonoResult} are satisfied. First, Lemma \ref{lem:MonotMoorePenosePseudoInv} implies that $\ker(A_{22})=\ker(B_{22})$ and hence $\operatorname{ran}(A_{22})=\operatorname{ran}(B_{22})$. Second, since $(B-A)^*=B-A, 0\leq B-A$ then it follows by Lemma \ref{lem:NecSuffCondPosMatrixUsingSchurCompl} that $0\leq (B-A)/(B-A)_{22},\ker(B_{22}-A_{22})\subseteq \ker(B_{12}-A_{12})$ and the latter implies that $\operatorname{ran}(B_{21}-A_{21})\subseteq \operatorname{ran}(B_{22}-A_{22})$. Thus, we have proved the hypotheses of Proposition \ref{prop:KerRanConditionForJPPTMonoResult} are satisfied. The proof of the theorem now follows immediately from this proposition and Corollary \ref{cor:NecSuffCondMatricesForMonotJPPT}.
\end{proof}

Before we move on, let us now compare our Proposition \ref{prop:KerRanConditionForJPPTMonoResult}, the statement that $(b)\Rightarrow (a)$ in Theorem \ref{thm:IntroMainThmPaperMonoGPPT}, and their proof in the self-adjoint case $A^*=A, B^*=B$ to Lemma 2.2 and Lemma 2.3, respectively, of Clement and Wimmer in \cite{01CW}. First, our Proposition \ref{prop:KerRanConditionForJPPTMonoResult} is a stronger result then \cite[Lemma 2.2]{01CW} since, although we both require that $\ker A_{22}=\ker B_{22}, \ker(B_{22}-A_{22})\subseteq \ker (B_{12}-A_{12})$ \{cf.\ (2.8) and (2.9) in \cite{01CW}\}, we don't require the hypotheses $\ker A_{22}\subseteq \ker A_{12}, \ker B_{22}\subseteq \ker B_{12}$ \{cf.\ (2.8) in \cite{01CW}\} in order to get the same conclusion (\ref{formula:EqDiffSchurComplAndDiffPPT}) \{cf.\ (2.10) in \cite{01CW}\}. Second, our Theorem \ref{thm:IntroMainThmPaperMonoGPPT} that $(b)\Rightarrow (a)$, is a stronger result then \cite[Lemma 2.3]{01CW} since we do not need to assume that $\ker B_{22}\subseteq \ker B_{12}$ in order to get the same conclusion (\ref{IntroMainThmPaperMonoSchurCompl}) \{cf.\ (2.14) in \cite{01CW}\}. Moreover, in essence, the main difference in our proof of Proposition \ref{prop:KerRanConditionForJPPTMonoResult} 
and that $(b)\Rightarrow (a)$ in Theorem \ref{thm:IntroMainThmPaperMonoGPPT} vs. the proof of Lemma 2.2 and Lemma 2.3 in \cite{01CW} is that we found a way to avoid using their hypothesis $\ker B_{22}\subseteq \ker B_{12}$ to achieve the same conclusions. The following example is useful to consider in regard to this discussion.
\begin{example}
Let $\mathbb{F}=\mathbb{R}$ or $\mathbb{F}=\mathbb{C}$. Consider the function $J\operatorname{ppt}(\cdot)$ on the following the following $2\times 2$ block matrices $A^*=A=[A_{ij}]_{i,j=1,2}, B^*=B=[B_{ij}]_{i,j=1,2}\in M_4(\mathbb{F})$,
\begin{gather}
    A=\left[\begin{array}{cc;{2pt/2pt}cc}
		0 & 0 & 1 & -\frac{1}{2} \\ 
            0 & 0 & 0 & 0 \\ \hdashline[2pt/2pt]
            1 & 0 & \frac{1}{2} & \frac{1}{2} \\ 
            -\frac{1}{2} & 0 & \frac{1}{2} & \frac{1}{2}
    \end{array}\right],\;
    B=\left[\begin{array}{cc;{2pt/2pt}cc}
		\frac{1}{2} & 0 & 1 & -\frac{1}{2} \\ 
            0 & 0 & 0 & 0 \\ \hdashline[2pt/2pt]
            1 & 0 & 1 & 1 \\ 
            -\frac{1}{2} & 0 & 1 & 1%
    \end{array}\right],\\
    J\operatorname{ppt}(A)=\left[\begin{array}{cc;{2pt/2pt}cc}
        -\frac{1}{8} & 0 & \frac{1}{4} & \frac{1}{4} \\ 
        0 & 0 & 0 & 0 \\ \hdashline[2pt/2pt]
        \frac{1}{4} & 0 & -\frac{1}{2} & -\frac{1}{2} \\ 
        \frac{1}{4} & 0 & -\frac{1}{2} & -\frac{1}{2}%
    \end{array}\right],\;
    J\operatorname{ppt}(B)=\left[\begin{array}{cc;{2pt/2pt}cc}
        \frac{7}{16} & 0 & \frac{1}{8} & \frac{1}{8} \\ 
        0 & 0 & 0 & 0 \\ \hdashline[2pt/2pt]
        \frac{1}{8} & 0 & -\frac{1}{4} & -\frac{1}{4} \\ 
        \frac{1}{8} & 0 & -\frac{1}{4} & -\frac{1}{4}%
    \end{array}\right].
\end{gather}
Then $A\leq B$ so that the hypotheses of Theorem \ref{thm:IntroMainThmPaperMonoGPPT} are satisfied and, in particular, we know by this theorem that $A/A_{22}\leq B/B_{22}$ and $J\operatorname{ppt}(A)\leq J\operatorname{ppt}(B)$ since 
\begin{gather}
    B_{22}^+=\begin{bmatrix}
        1 & 1 \\ 
        1 & 1
    \end{bmatrix}^+=\begin{bmatrix}
        \frac{1}{4} & \frac{1}{4} \\ 
        \frac{1}{4} & \frac{1}{4}
    \end{bmatrix}
    \leq \begin{bmatrix}
        \frac{1}{2} & \frac{1}{2} \\ 
        \frac{1}{2} & \frac{1}{2}
    \end{bmatrix}=\begin{bmatrix}
        \frac{1}{2} & \frac{1}{2} \\ 
        \frac{1}{2} & \frac{1}{2}
    \end{bmatrix}^+=A_{22}^+.
\end{gather}
On the other hand, we cannot use Lemma 2.3 in \cite{01CW} to prove that $A/A_{22}\leq B/B_{22}$ since
\begin{gather}
    \ker(B_{22})=\operatorname{span}\left\{\begin{bmatrix}
        -1\\
        1
    \end{bmatrix}\right\}\not\subseteq \operatorname{span}\left\{\begin{bmatrix}
        1\\
        2
    \end{bmatrix}\right\}=\ker(B_{12}).
\end{gather}
\end{example}

We now complete the proof of Theorem \ref{thm:IntroMainThmPaperMonoGPPT} by proving that statements $(b)$ and $(c)$ are equivalent. Before we do this we need the following results from \cite[Observation (7)]{85JR} (cf.\ \cite[Observation 3.1]{95JT}) and \cite[Theorem 2.1]{93HN} (cf.\ \cite[Theorem 5]{91YS}) which is complementary to Lemma \ref{lem:MonotMoorePenosePseudoInv} above.

\begin{lemma}\label{lem:MonotMPInvSpectralCharacterization}
    If $C^*=C, D^*=D\in M_m(\mathbb{F}),$ and $D$ is invertible then
    \begin{gather}
        \det[(1-t)C+tD]\not=0,\;\forall t\in[0,1] \iff \sigma (D^{-1}C)\cap (-\infty,0]=\emptyset,
    \end{gather}
    where $\sigma(D^{-1}C)$ denotes the set of all eigenvalues of $D^{-1}C$. Furthermore, if $C\leq D$ then $\sigma (D^{-1}C)\subseteq \mathbb{R}$ and
        \begin{gather}
        \sigma (D^{-1}C)\cap (-\infty,0]=\emptyset \iff D^{-1}\leq C^{-1}.
    \end{gather}
\end{lemma}

\begin{proof}[Proof of $(b)\Leftrightarrow (c)$ in Theorem \ref{thm:IntroMainThmPaperMonoGPPT}]
    Assume the hypotheses in Theorem \ref{thm:IntroMainThmPaperMonoGPPT}. We begin by proving $(c)\Rightarrow (b)$ in Theorem \ref{thm:IntroMainThmPaperMonoGPPT}. To do this we can assume without loss of generality that $\mathbb{F}=\mathbb{C}$. Observe that, since $A_{22}^*=A_{22}, B_{22}^*=B_{22}, 0\leq B_{22}-A_{22}$, then it follows that the holomorphic $n_{2}\times n_{2}$-matrix-valued function $H:\mathbb{C}^+\rightarrow M_{n_2}(\mathbb{C})$ defined by
    \begin{gather*}
        H(t)=(1-t)A_{22}+tB_{22}=A_{22}+t(B_{22}-A_{22}),\;t\in \mathbb{C}
    \end{gather*}
    is a matrix-valued Herglotz-Nevanlinna function on the open upper half of the complex plane $\mathbb{C}^+=\{t\in \mathbb{C}:\operatorname{Im}z>0\}$. Let
    \begin{gather*}
        r=\operatorname{rank} H(i).
    \end{gather*}
    In the case $r=0$, it follows that $H(i)=0$ and hence $A_{22}=B_{22}=0$ so that $(b)\Leftrightarrow (c)$ in Theorem \ref{thm:IntroMainThmPaperMonoGPPT} is true in this case. Now assume that $r\not=0$. Then by well-known theorems on matrix-valued Herglotz-Nevanlinna functions (see \cite{00GT, 12FKa, 12FKb, 17FK}), it follows that there exists a unitary matrix $U\in M_{n_2}(\mathbb{C})$ and matrices $H_0^*=H_0, H_1^*=H_1\in M_{r}(\mathbb{C})$ with $0\leq H_1$ such that
    \begin{gather*}
        H(t)=U^*\begin{bmatrix}
            h(t) & 0\\
            0 & 0
        \end{bmatrix}U,\;\;h(t)=H_0+tH_1,\;t\in \mathbb{C}, 
    \end{gather*}
    the set
    \begin{gather*}
        Z=\{t\in \mathbb{C}:\operatorname{rank} H(t)\not=r\}=\{t\in \mathbb{C}:\det h(t)=0\}
    \end{gather*}
    contains a finite number of elements with $Z\subseteq \mathbb{R}$, and the function $H^+:\mathbb{C}\setminus Z\rightarrow  M_{n_2}(\mathbb{C})$ defined by
    \begin{gather*}
        H^+(t)=H(t)^+=U^*\begin{bmatrix}
            h(t)^{-1} & 0\\
            0 & 0
        \end{bmatrix}U,\;t\in \mathbb{C}\setminus Z
    \end{gather*}
    is a holomorphic $n_{2}\times n_{2}$-matrix-valued function (in  which $-H^+$ is a matrix-valued Herglotz-Nevanlinna function on $\mathbb{C}^+$). Moreover, it follows that $H^+$ is continuously differentiable on $\mathbb{R}\setminus Z$ with $H^+(t)^*=H^+(t)$ for all $t\in \mathbb{R}\setminus Z$ and
    \begin{gather*}
        (H^+)'(t)=\frac{d}{dt}H(t)^+=U^*\begin{bmatrix}
            \frac{d}{dt}h(t)^{-1} & 0\\
            0 & 0
        \end{bmatrix}U=U^*\begin{bmatrix}
            -h(t)^{-1}h'(t)h(t)^{-1} & 0\\
            0 & 0
        \end{bmatrix}U\\
        =-H^+(t)H'(t)H^+(t)=-H^+(t)(B_{22}-A_{22})H^+(t)\leq 0,\;\;t\in \mathbb{R}\setminus Z.
    \end{gather*}
    Hence, it follows immediately from this that $\operatorname{rank}H(t)$ is constant on an interval $[a,b]\subseteq \mathbb{R}$ if and only if $[a,b]\subseteq \mathbb{R}\setminus Z$, in which case
    \begin{gather*}
        H(b)^+=H(a)^++\int_{a}^b (H^+)'(t)dt\leq H(a)^+.
    \end{gather*}
    This proves the statement $(c)\Rightarrow (b)$ in Theorem \ref{thm:IntroMainThmPaperMonoGPPT}.

    We will now prove the statement $(b)\Rightarrow (c)$ in Theorem \ref{thm:IntroMainThmPaperMonoGPPT}. Suppose that $(B_{22})^+\leq (A_{22})^+$. If $A_{22}=B_{22}=0$ then $(b)\Rightarrow (c)$ in Theorem \ref{thm:IntroMainThmPaperMonoGPPT} in this case. Hence, suppose $A_{22}$ and $B_{22}$ are not both the zero matrix. Then by Lemma \ref{lem:MonotMoorePenosePseudoInv} it follows that $0\not= r_0=\operatorname{rank}A_{22}=\operatorname{rank}B_{22}$ and there exists a matrix $V\in M_{n_2}(\mathbb{\mathbb{F}})$ with $V^*=V^{-1}$ and invertible self-adjoint matrices $C,D\in M_{r_0}(\mathbb{F})$ with $0\leq D-C$ such that 
    \begin{gather*}
        A_{22}=V^*\begin{bmatrix}
            C & 0\\
            0 & 0
        \end{bmatrix}V,\; B_{22}=V^*\begin{bmatrix}
            D & 0\\
            0 & 0
        \end{bmatrix}V.
    \end{gather*}
    It follows from this that
    \begin{gather*}
        H(t)=(1-t)A_{22}+tB_{22}=V^*\begin{bmatrix}
            (1-t)C+tD & 0\\
            0 & 0
        \end{bmatrix}V.
    \end{gather*}
    Thus, to prove the statement $(c)$ in Theorem \ref{thm:IntroMainThmPaperMonoGPPT}, we need only show that $\det[(1-t)C+tD]\not=0$ for all $t\in [0,1]$. But this follows immediately from Lemma \ref{lem:MonotMPInvSpectralCharacterization} since by the hypothesis that $(B_{22})^+\leq (A_{22})^+$ we must have $D^{-1}\leq C^{-1}$.
\end{proof}

\section*{Acknowledgements}

The authors would like to thank Joseph A.\ Ball for bring to our attention the Potapov-Ginzburg transform in relation to the principal pivot transform. The second author would also like to thank Graeme W.\ Milton and Maxence Cassier for all the helpful discussions that motivated this paper and which lead him to find the reference \cite{00CL}.

\section*{Funding:} This research did not receive any specific grant from funding agencies in the public, commercial, or not-for-profit sectors.

\section*{Declarations of interest:} None.

\bibliographystyle{abbrvnat}
\bibliography{mybibliography}

\end{document}